\documentclass{amsart} 
\usepackage{amssymb} 
\theoremstyle{plain} 
\newtheorem{theorem}{\indent\sc Theorem}[section]
\newtheorem{lemma}[theorem]{\indent\sc Lemma}

\theoremstyle{definition} 

\begin{document}

\title{On the most expected number of components for random links} 

\author[Kazuhiro Ichihara]{Kazuhiro Ichihara$^*$} 

\author[Ken-ichi Yoshida]{Ken-ichi Yoshida$^\dagger$} 


%
\thanks{ 
$^{*}$Partially supported by JSPS KAKENHI Grant Number 26400100. 
$^\dagger$Partially supported by JSPS KAKENHI Grant Number 25400050. 
$^{*}$\,$^\dagger$Partially supported by Joint Research Grant of Institute of Natural Sciences at Nihon University for 2015. 
}
\address{
   	Department of Mathematics \endgraf
	College of Humanities and Sciences \endgraf
	Nihon University \endgraf
	3-25-40 Sakurajosui, Setagaya-ku, Tokyo 156-8550 \endgraf
	Japan
	}
\email{ichihara@math.chs.nihon-u.ac.jp}

\email{yoshida@math.chs.nihon-u.ac.jp}


\maketitle

\begin{abstract}
We consider a random link, 
which is defined as the closure of a braid 
obtained from a random walk on the braid group. 
For such a random link, 
the expected value for the number of components 
was calculated by Jiming Ma. 
In this paper, we determine 
the most expected number of components 
for a random link, 
and further, consider 
the most expected partition of the number of strings for a random braid. 
\end{abstract}

\section{Introduction} 

In \cite{Ma13}, from a probabilistic point of view, 
Jiming Ma introduced and studied two models of random links. 
We here consider the one which is defined as 
the braid closures of randomly chosen braids via random walks on the braid groups. 

Suppose that such a random walk 
on the braid group $\mathfrak{B}_n$ of $n$-strings 
induces the uniform distribution on 
the symmetric group $\mathfrak{S}_n$ on $n$ letters 
via the natural projection $\mathfrak{B}_n \to \mathfrak{S}_n$ ($n \ge 3$). 
Then, Ma showed in \cite[Theorem 1.1]{Ma13} that, 
for the random link coming from a random walk of $k$-step 
on $\mathfrak{B}_n$ ($n \ge 3$), 
the expected value of the number of components converges to 
\[ 1+ \frac{1}{2} + \frac{1}{3} + \cdots + \frac{1}{n} \]
when $k$ diverges to $\infty$. 
See the next section for the precise definition of the random link. 

From this result, it is natural to ask what is 
the most expected number of components for such a random link. 
We first answer this question as follows. 

\begin{theorem}\label{Thm1}
Consider a random link obtained from a random walk on $\mathfrak{B}_n$. 
Suppose that the random walk on $\mathfrak{B}_n$ 
is defined for the probability distribution on $\mathfrak{B}_n$ 
which induces the uniform distribution on $\mathfrak{S}_n$ 
via the natural projection $\mathfrak{B}_n \to \mathfrak{S}_n$ $(n \ge 3)$. 
Then the most expected number of components is equal to 
\[
K_n = \left[ \log(n+1) + \gamma -1
+ \frac{\zeta(2)-\zeta(3)}{\log(n+1)+ \gamma - 1.5} 
+ \frac{h}{(\log(n+1)+ \gamma - 1.5)^2} \right]
\]
where $[x]$ denotes the integer part of $x$, 
$\zeta$ is the Riemann zeta function, 
$\gamma = 0.5772\dots$ is the Euler-Mascheroni constant, 
and $h$ with $-1.1 < h < 1.5$ is a function on $n$, i.e, $h = h(n)$.
In particular, if $n >188$, it follows that 
\[ \left[\;\log n - \frac{1}{2}\;\right] < K_n < \left[\;\log n\;\right]\,. \]
\end{theorem}

In fact, this can be obtained from 
some known results on Combinatorics and Analytic number theory.

To connect the problem of random link to them, 
the key is the correspondence between 
components of the closure of a braid 
and 
cycles in the cycle decomposition 
of the permutation corresponding to the braid. 
In particular, 
the number of components are calculated 
as the number of cycles. 

In view of this, we can relate 
random braids to random partitions of integers (i.e., the numbers of strings). 
Then it is also natural to ask what is 
the most expect partition of the number of strings for a random braid. 
About this question, against our naive intuition, we can show the following. 

\begin{theorem}\label{Thm2}
Consider a random braid obtained from a random walk on $\mathfrak{B}_n$. 
Suppose that the random walk on $\mathfrak{B}_n$ 
is defined for the probability distribution on $\mathfrak{B}_n$ 
which induces the uniform distribution on $\mathfrak{S}_n$ 
via the natural projection $\mathfrak{B}_n \to \mathfrak{S}_n$ $(n \ge 3)$. 
Then the most expected partition of the number of the strings is $( (n-1), 1) $. 
\end{theorem}

Actually the probability for such a partition of the number of the strings is shown to converge to $1/(n-1)$. 

\bigskip

The first author thanks Jiming Ma for useful discussions in this topic, 
and also thanks Kazuma Shimomoto for letting him know about the Stirling number of the first kind.

\section{Link, braid and random walk}

We here give a brief review of the setting for studying the random links introduced in \cite{Ma13}. See \cite{Ma13} for details. 

Throughout the paper, 
we denote the braid group of $n$-strings by $\mathfrak{B}_n$, 
and the symmetric group on $n$ letters by $\mathfrak{S}_n$. 

We consider a probability distribution $\mu$ on $\mathfrak{B}_n$. 
By using such a probability distribution, 
one can define a random walk 
by setting the transition probability as $\mathbb{P} (x,y) = \mu (x y^{-1})$. 
Here we suppose that our random walk starts at the identity element at time zero. 

By considering the natural projection $\mathfrak{B}_n \to \mathfrak{S}_n$, 
such a random walk on $\mathfrak{B}_n$ induces 
a random walk on $\mathfrak{S}_n$. 
We here suppose that the probability distribution $\mu$ 
induces the uniform distribution on $\mathfrak{S}_n$ via the natural projection. 
Here, by the \textit{uniform distribution} on $\mathfrak{S}_n$, 
we mean the probability distribution satisfying 
$\mathbb{P} (s) = 1/n!$ holds for any $s \in \mathfrak{S}_n$. 
That is, we are assuming that the probability $\mathbb{P} (s)$ 
for any $s \in \mathfrak{S}_n$ 
induced from the random walk is sufficiently close to $1/n!$, 
or, in other words, 
the induced random walk on $\mathfrak{S}_n$ 
is the uniformly distributed random walk. 

Then, conceptually, we said a braid is a \textit{random braid} 
if it is represented by a braid coming from a random walk on $\mathfrak{B}_n$ with sufficiently long steps. 

We remark that our assumption on the probability distribution does not give severe restriction. 
Actually, Ma showed the following as \cite[Theorem 2.5]{Ma13}. 
Let $\mu$ be a probability distribution on $\mathfrak{B}_n$, 
which induces 
a random walk $\overline{ \omega_{n,k} }$ on $\mathfrak{S}_n$.
Suppose that 
the probability $\mathbb{P} ( \overline{ \omega_{n,1} } = e )$ is larger than 0, 
for the identity element $e \in \mathfrak{S}_n$, 
and the support of $\mu$ generates $\mathfrak{B}_n$. 
Then $\mu$ induces the uniform distribution on $\mathfrak{S}_n$. 

For example, 
the probability distribution $\mu_c$ on $\mathfrak{B}_n$ defined by
$$
\mu_c ( e) = \mu_c ( \sigma_i ) = \mu_c ( \sigma_i^{-1} ) 
= \frac{1}{2n-1} $$ 
for the identity element $e$ and each canonical generator 
$\sigma_i \in \mathfrak{B}_n$ ($1 \le i \le n-1$) 
is shown to satisfy the assumption.

Now we consider a random walk $\omega_{n,k}$ on $\mathfrak{B}_n$, 
and the probability $p_{n,k}^m $ for the link 
corresponding to the random walk $\omega_{n,k}$ 
which has exactly $m$ components. 
Then, for the random link, 
we say that \textit{the most expected number of components is $m$} 
if, for any sufficiently large $k$, $p_{n,k}^m$ is maximal among $p_{n,k}^j$ for $1 \le j \le n$.

\section{Most expected number of components}

In this section, we give a proof of Theorem \ref{Thm1}.

\begin{proof}[Proof of Theorem \ref{Thm1}]
Consider a random walk $\omega_{n,k}$ on $\mathfrak{B}_n$. 
By taking the braid closure of $\omega_{n,k}$, 
we have a link $\widehat{\omega_{n,k}}$ in the 3-sphere. 

Consider the natural projection $\pi : \mathfrak{B}_n \to \mathfrak{S}_n$. 
We see that a component of $\widehat{\omega_{n,k}}$ 
corresponds to an orbit of the action of $\pi(\omega_{n,k})$ on $n$ letters. 
It follows that, 
if we consider the decomposition of $\pi(\omega_{n,k})$ 
into cycles with mutually distinct letters, 
the number of components of $\widehat{\omega_{n,k}}$ 
is equal to the number of cycles in the decomposition of $\pi(\omega_{n,k})$. 

Now we are supposing that $\omega_{n,k}$ is defined by 
a probability distribution on $\mathfrak{B}_n$ 
which induces the uniform probability distribution on $\mathfrak{S}_n$ 
via the natural projection $\pi : \mathfrak{B}_n \to \mathfrak{S}_n$. 
This means that, for any $s \in \mathfrak{S}_n$, 
the probability $\mathbb{P} (s)$ defined by 
the induced random walk $\pi ( \omega_{n,k} )$ converges to $1/n!$. 

Let $p_{n,k}^m$ be the probability for the link $\widehat{\omega_{n,k}}$ 
corresponding to the random walk $\omega_{n,k}$ 
which has exactly $m$ components. 
It then follows that, as $k \to \infty$, $p_{n,k}^m$ converges to the ratio of 
the number of permutations with disjoint $m$ cycles in $\mathfrak{S}_n$. 

Here we note that 
the number of permutations of $n$ letters with disjoint $m$ cycles 
is called \textit{the Stirling number of the first kind}, denoted by $c(n,m)$. 
Consequently, 
to obtain the most expected number of components for $\widehat{\omega_{n,k}}$, 
it suffices to study the value of $m$ 
for which $c(n,m)$ is maximal for $1 \le m \le n$. 

This was already established by Hammersley in \cite{Hammersley} that 
$c(n,m)$ is maximal for $1 \le m \le n$ if $m$ is equal to 
$$
K_n = \left[ \log(n+1) + \gamma -1
+ \frac{\zeta(2)-\zeta(3)}{\log(n+1)+ \gamma - 1.5} 
+ \frac{h}{(\log(n+1)+ \gamma - 1.5)^2} \right]
$$
where $[x]$ denotes the integer part of $x$, 
$\zeta$ is the Riemann zeta function, 
$\gamma = 0.5772\dots$ is the Euler-Mascheroni constant, 
and $h$ with $-1.1 < h < 1.5$ is a function on $n$, i.e, $h = h(n)$.

Furthermore, if $n >188$, Erd\"os proved in \cite{Erdos} that 
$$ \left[ \log n - \frac{1}{2} \right] < K_n < \left[ \log n \right]$$
holds. 

This completes the proof of Theorem \ref{Thm1}. 

\end{proof}

\section{Partition of the number of strings for braid}

In this section, we give a proof of Theorem \ref{Thm2}. 
Before starting the proof, we should fix our terminology. 

An element of the symmetric group $\mathfrak{S}_n$ on $n$ letters 
is uniquely represented as a composition of several cycles with distinct letters. 
The set of the lengths of such cycles gives a partition of the integer $n$. 
That is, if an element of $\mathfrak{S}_n$ is represented as a composition of cycles of lengths $n_1, n_2, \cdots, n_m$ with $n_1 \ge n_2 \ge \cdots \ge n_m$, then we have a partition $(n_1, n_2, \cdots, n_m)$ of $n$, 
for $n = n_1 + n_2 + \cdots + n_m$ holds. 

In view of this, given a braid $\sigma$ with $n$-strings with $n > 0$, 
we define \textit{a partition of the number of strings} for $\sigma$ as a non-increasing sequence of positive integers $(n_1, n_2, \cdots , n_m)$ which is obtained in that way for the element $\pi (\sigma)$ of $\mathfrak{S}_n$, where $\pi$ denotes the natural projection $\mathfrak{B}_n \to \mathfrak{S}_n$. 

\medskip

We here prepare the following, 
which is the key algebraic lemma to prove Theorem \ref{Thm2}. 

\begin{lemma}
In the symmetric group on $n$ letters with $n \ge 3$, 
the conjugacy class of the maximal cardinality is the one containing the $(n-1)$-cycle $(1 \ 2 \ \dots \ n-1)$, and the cardinality is $n \cdot (n-2)!$. 
\end{lemma}

\begin{proof}
Let $\mathfrak{S}_n$ be the symmetric group on $n$ letters $(n \ge 3)$. 

It is known that the cardinality of the conjugacy classes including $a \in \mathfrak{S}_n$ 
is given by $|\mathfrak{S}_n| / |Z(a)|$ (see \cite[Chapter 6, pp.198]{Ar} for example), 
where $Z(a)$ denotes the centralizer of $a$, that is, $\{ g \in \mathfrak{S}_n | ga=ag \}$.

Thus it suffice to show that $|Z(a)| \ge n-1$ for any element $a \in \mathfrak{S}_n$. 

We first claim that, in general, $k_1\cdots k_r \ge k_1+\cdots + k_r$ holds for a tuple of integers $k_1,...,k_r \ge 2$. 
This is easily shown by induction, and the equality holds only when $r=1$, or $r=2$ and $k_1=k_2=2$.

Now let us describe $a \in \mathfrak{S}_n$  by a product of cycles without common letters: 
for example, $a=a_1 \cdots a_r$ with $a_i$ is a $k_i$-cycle and $k_1 \ge k_2 \ge \cdots \ge k_r \ge 1$. 

Here we note that, if $k_r \ge 2$, then the centralizer $Z(a)$ contains the direct product of abelian groups generated by $a_1,\ldots,a_r$. 
Thus the order of $Z(a)$ is at least $k_1\cdots k_r$, which is greater than or equal to $k_1+\cdots + k_r=n$ by the above claim. 

If $k_{r-1} \ge 2, k_r =1$, then $Z(a)$ contains the direct product of abelian groups generated by $a_1,\ldots,a_{r-1}$, which has $k_1\cdots k_{r-1}$ elements. 
Again, by the above claim, the order of $Z(a)$ is at least $k_1+ \cdots + k_{r-1} = n-1$. 

Finally, if $k_p=\cdots =k_r =1$ for some $p$ with $2 \le p \le r-1$, then $Z(a)$ contains the direct product of abelian groups generated 
by $a_1,\ldots,a_{p-1}$ and the $(r-p+1)$-cycle of the other letters. 
The order of the cycle is at least $k_1\cdots k_{p-1}\cdot (r-p+1) \ge k_1+\cdots + k_{p-1}+(r-p+1)=n$.

Consequently we see that $|Z(a)| \ge n-1$. 

Furthermore, suppose that $|Z(a)|=n-1$ holds, then we have $r-1=1$, $k_1 \ge 2$, and $k_2=1$, 
that is, $a$ is the $(n-1)$-cycle, or $r-1=2, k_1=k_2=2,k_3=1,n-1=4$. 
In the latter case, i.e., $n=5$, we may assume that $a=(12)(34)$. 
But in this case, we have $|Z(a)|=8$ (the quaternion group is also contained), and so, the equality does not hold. 

\end{proof}

\begin{proof}[Proof of Theorem \ref{Thm2}]
Consider a random walk $\omega_{n,k}$ on $\mathfrak{B}_n$. 
By using the natural projection, 
we have the induced random walk on $\mathfrak{S}_n$. 
Since we have assumed that 
this induced random walk is uniformly distributed, 
the probability of a braid in the sequence with a given partition, 
say $(n_1, n_2, \cdots, n_m)$, of the number $n$ of the strings converges to $\Delta_{n,i} / n!$, 
where $\Delta_{n,i}$ denotes 
the number of elements in $\mathfrak{S}_n$ 
giving that partition of the integer $n$. 
This $\Delta_{n,i}$ is equal to the cardinality of the conjugacy class of an element in $\mathfrak{S}_n$ decomposed into the cycles of distinct letters of lengths $n_1, n_2, \cdots, n_m$. 
Then, by the above lemma, $\Delta_{n,i}$ takes maximum for the one containing the $(n-1)$-cycle $(1 \ 2 \ \dots \ n-1)$. 
That is, the most expected partition of the number of the strings for a random $n$-braid must be $( (n-1), 1) $. 
Also, since the maximum of $\Delta_{n,i}$ is $n \cdot (n-2)!$, 
the most expected probability is $n \cdot (n-2)! / n! = 1/(n-1)$. 
Furthermore, in that case, the link comes from the braid corresponding to the $(n-1)$-cycle $(1 \ 2 \ \dots \ n-1)$. 
\end{proof}

Actually, in the same way, it can be shown that the probability that a given random link becomes a knot converges to $1/n$.

%
%

\end{document}